\documentclass[11pt]{amsart}%
\usepackage[colorlinks=true]{hyperref}
\usepackage{enumerate}
\usepackage{amsfonts}
\usepackage{amsmath}
\usepackage{amssymb}
\usepackage{graphicx}%
\usepackage[final]{optional}
\usepackage[centertags]{amsmath}%
\usepackage{amsfonts,amssymb,amsthm, stmaryrd , graphicx, mathrsfs, mathpazo}%
\usepackage[usenames]{color}
\usepackage[toc,page]{appendix}

\newtheorem{thm}{Theorem}[section]
\newcommand{\bt}{\begin{thm}}
\newcommand{\et}{\end{thm}}

\newtheorem{cor}[thm]{Corollary} 
\newcommand{\bc}{\begin{cor}}
\newcommand{\ec}{\end{cor}}
\newtheorem{lem}[thm]{Lemma}  
\newcommand{\bl}{\begin{lem}}
\newcommand{\el}{\end{lem}}
\newtheorem{prop}[thm]{Proposition}
\newcommand{\bp}{\begin{prop}}
\newcommand{\ep}{\end{prop}}
\newtheorem{defn}[thm]{Definition}
\newcommand{\bd}{\begin{defn}}  
\newcommand{\ed}{\end{defn}}

\newtheorem{rmrk}[thm]{Remark}
\newcommand{\br}{\begin{rmrk}}
\newcommand{\er}{\end{rmrk}}

\newcommand{\be}{\begin{equation}}
\newcommand{\ee}{\end{equation}}

\newcommand{\R}{\mathbb{R}}

\newcommand{\tr}{\operatorname{tr}}

\begin{document}
\title[Differential Harnack Estimates for heat Equation under Finsler-Ricci Flow]{Differential Harnack Estimates for Positive Solutions to Heat Equation under Finsler-Ricci Flow}

\author{Sajjad Lakzian}
\address[Sajjad Lakzian]{HCM, Universit\"{a}t Bonn}
\email{lakzians@gmail.com}

\thanks{\textit{The author is supported by the Hausdorff postdoctoral fellowship at the Hausdorff Center for Mathematics, University of Bonn, Germany}}

\subjclass[2010]{35K55(primary), and 53C21(secondary)} 
 \keywords{Finsler Ricci Flow, Differential Harnack, Gradient Estimate, Weighted Ricci Curvature, Heat Equation, Curvature-Dimension}

\begin{abstract}
In this paper we prove first order differential Harnack estimates for positive solutions of the heat equation (in the sense of distributions) under closed Finsler-Ricci flows. We assume suitable Ricci curvature bounds throughout the flow and also assume that $\mathbf{S}-$curvature vanishes along the flow. One of the key tools we use is the Bochner identity for Finsler structures proved by Ohta-Sturm~\cite{Ohta-Sturm-Bochner}.
\end{abstract}

\maketitle






\section{Introduction}

In the past few decades, geometric flows and more notably among them, Ricci flow, have proved very useful in attacking long standing geometry and topology questions. One important application is finding the so called round (of constant curvature, Einstein , Solitons, .. etc.) metrics on manifolds by homogenizing a given initial metric.

There is also a hope that similar methods can be applied in the Finsler setting. One might hope to  find an answer for, for instance, Professor Chern's question about the existence of Finsler-Einstein metrics on every smooth manifolds by using a suitable geometric flow resembling the Ricci flow.

In the Finsler setting, there are notions of Ricci and sectional curvatures and Bao~\cite{Bao-curv-problems} has proposed an evolution of Finsler structures that in essence shares a great resemblance with the Ricci flow of Riemannian metrics. The flow Bao suggests is $\frac{\partial F^2}{\partial t} = -2F^2 R$ where $R = \frac{1}{F^2} Ric$. In terms of the symmetric metric tensor associated to $F$ and Akbarzadeh's Ricci tensor, this flow takes the form of $\frac{\partial g_{ij}}{\partial t} = -2Ric_{ij}$ which is the familiar Ricci flow. 

The notion of Finsler-Ricci flow is very recent and very little has been done about it. Some partial results regarding the existence and uniqueness of such flows are obtained in~\cite{RAZAVI-FRF}. Also, the solitons of this flow have been studied in \cite{Bidabad-Yarahmadi}. Our focus in these notes will be to consider a positive solution of the heat equation (in the sense of distributions) under Finsler-Ricci flow and prove first order differential Harnack estimates that are similar to those in the Riemannian case (see~\cite{LIU-Grad-Est} ~\cite{SUN-Grad-Est}). The key tools we have used here is the Bochner identity for Finsler metrics (point-wise and in the sense of distributions) proven by Ohta-Sturm ~\cite{Ohta-Sturm-Bochner} and as is customary in such estimates, the Maximum principle. 

We should mention that, in this paper, we are not dealing with the existence and Sobolev regularity of such solutions (which is very important and extremely delicate, for example, in the static case, solutions will be $C^2$ iff the structure is Riemannian). For existence and the regularity in the static case see Ohta-Sturm~\cite{Ohta-Sturm-Heat}. Our main theorem is the following
\begin{thm}\label{thm-main}
  Let $\left( M^n , F(t) \right) , t \in [0,T]$ be a closed Finsler-Ricci flow. Suppose there is a real number $K \in \R$ and positive real numbers $K_1$ and $K_2$ such that for all $t \in [0,T]$,\\
  
\noindent \textbf{(i)} $- K_1  \le \left( Ric_{ij}(\mathbf{v}) \right)_{i,j=1}^n \le K_2 $ as quadratic forms on $T_x M$ for all $v \in T_x M \setminus \{0\}$,
in any coordinate system, $\{\frac{\partial}{\partial x_i}\}$, that is orthonormal with respect to $g_{\mathbf{v}}$. \\
\noindent \textbf{(ii)} $\mathbf{S}-$curvature vanishes (see Section~\ref{sec:S-curv});\\
 
Let $u(x,t) \in L^2 \left([0,T] , H^1(M) \right) \cap H^1 \left([0,T] , H^{-1}(M)  \right)$ be a positive \textbf{global solution} (in the sense of distributions) of the heat equation under FRF, i.e. for any test function, $\phi \in C^\infty(M)$, and for all $t \in [0,T]$,
\be
	\int	_M \phi \partial_t u(t,.) \; dm =  - \int_M D\phi \left( \nabla u(t,.)  \right) \; dmdt ;
\ee
Then, $u$ satisfies,
\be
		 F^2\left( \nabla \left( \log u \right) (t,x)\right) - \theta \partial_t \left( \log u \right)(t,x) \le \frac{n \theta^2}{t} + \frac{n\theta^3 C_1}{\theta -1} + n^{3/2}\theta^2 \sqrt{C_2},
\ee
for any $\theta>1$ and where,
\be
C_1 = K_1   \;\; and \;\; C_2 = \max\{K_1^2 , K_2^2\}.
\ee
\end{thm}

\begin{rmrk}
Our results can be applied to any Finsler-Ricci flow of Berwald metrics on closed manifolds, since the $\mathbf{S}-$curvature vanishes for Berwald metrics (for example see~\cite{Ohta-vanishing}).
\end{rmrk}

We will note that it might be possible to obtain stronger results with less curvature bound conditions by using different methods such as Nash-Moser iteration (as is done for harmonic functions in static case by Xia~\cite{Xia-grad-est}). 

Integrating the differential Harnack inequalities, in an standard manner, lead to the Harnack type inequalities, indeed,
\begin{cor}\label{cor-global}
	Let $\left( M , F(t) \right) , t \in [0,T]$ be as in Theorem~\ref{thm-main}, then for any two points $(x , t_1) , (y , t_2) \in M \times (0, T]$ with $t_1 < t_2$, we will get,
	\be
		u(x,t_1) \le u(y,t_2) \left(\frac{t_2}{t_1}  \right)^{ 2n\epsilon}  \exp \left\{ \int_0^1 \frac{\epsilon F^2\left( \gamma'(s) \right)|_\tau}{2(t_2 - t_1)}ds + C(n,\epsilon) (t_2 - t_1)\left( C_1 + \sqrt{C_2}\right) \right\},
	\ee
\end{cor}
whenever $\epsilon > \frac{1}{2}$ and for $C$, depending on $n$ and $\epsilon$ only and where the dependencies of $C_1$ and $C_2$ on our parameters are as in Theorem~\ref{thm-main}; $\gamma$ is a curve joining $x,y$ with $\gamma(1) = x$ and $\gamma(0) = y$, and $F\left( \gamma'(s) \right)|_\tau$ is the speed of $\gamma$ at time $\tau = (1-s)t_2 + st_1$.

The organization of this paper is as follows: We first briefly review some facts and results about differential Harnack estimates in Riemannian setting and about Finsler geometry in Section~\ref{sec-background}; In section~\ref{sec-estimates}, we will present lemmas and computations that we need in order to obtain a useful parabolic partial differential inequality and in the last section, we will complete the proof of our main theorem.

In this paper, we will sometimes use the abbreviations, DHE for \textit{"Differential Harnack Estimates"} and FRF for Bao's \textit{"Finsler-Ricci Flow"}






\section*{Acknowledgements}

The author would like to thank Professor K-T Sturm for giving the author the opportunity of working under his supervision as a postdoctoral fellow and for his great insights. The author also would like to thank Professor Christina Sormani for her constant support and encouragement. Many thanks go to the Stochastic Analysis group in Bonn and also my colleagues Mike Munn and Lashi Bandara.

The author is deeply grateful for valuable comments made and numerous corrections suggested by the anonymous reviewer of this paper. 






\section{Background}\label{sec-background}




\subsection{DHE estimates for heat equation in Riemannian Ricci flow}

The Ricci flow equation , $\frac{\partial g}{\partial t} = -2Ric$, was first proposed by Richard Hamilton in his seminal paper~\cite{Hamilton-Ricci-Flow}. Ricci flow is a heat type quasilinear pde but as is well-known, it enjoys a short time existence and uniqueness theory (see~\cite{Hamilton-Ricci-Flow}) and has been the key tool in proving the Poincare and Geometrization conjectures. 

The gradient estimates for solutions of parabolic equations under Ricci flow are a very important part of the Ricci flow theory. Perelman in his ground breaking work~\cite{Perelman-Entropy} proves such estimates for the conjugate heat equation which then he would benefit from, in the analysis of his $\mathcal{W}-$entropy functional. Since then there have been many important results in this direction (for both heat equation and conjugate heat equation), for example in~\cite{Kuang-Zhang}~\cite{Bailesteanu-etal}~\cite{Cao-etal-diffharnack}~\cite{Cao-Hamilton-diffharnack}~\cite{Cao-diffharnack}, to name a few. 

Since our proof, in the spirit, is closer to ones in~Liu~\cite{LIU-Grad-Est} and Sun~\cite{SUN-Grad-Est}, we will only mention their result without commenting on the other literature in this direction. Their estimates for positive solution of the heat equation under a closed Ricci flow can be stated follows\\

\textit{\textbf{Theorem [LIU~\cite{LIU-Grad-Est} , SUN~\cite{SUN-Grad-Est}] } Let $\left(M , g(t) \right) ; t \in [0,T]$  be a closed Ricci flow solution with $ - K_1 \le Ric \le K_2$ ($K_1,K_2 >0$)  along the flow. for, $u(x,t)$, a positive solution of the heat equation $\left( \Delta_{g(t)} - \partial_t \right) u(x,t) = 0$, one has the following first order gradient estimates
\be
	\frac{ \left| \nabla u(x,t) \right|^2}{u^2(x,t)} - \theta \frac{\partial_t u(x,t)}{u(x,t)} \le \frac{n\theta^2}{t} + \frac{n \theta^3 K_1}{\theta -1} + n^{\frac{3}{2}}\theta^2\left( K_1 + K_2 \right),
\ee
where, $\theta>1$.}

Their method of proof is to take $ f = \log u$ and
\be
	\alpha := t \left( \frac{ \left| \nabla u(x,t) \right|^2}{u^2(x,t)} - \theta \frac{\partial_t u(x,t)}{u(x,t)} \right) = t \left( \left| \nabla f \right|^2 - \theta \partial_t f  \right),
\ee
and apply the maximum principle to the parabolic partial differential inequality
\begin{eqnarray*}
	\left( \Delta_{g(t)} - \partial_t \right) \alpha + 2 Df \left(\nabla \alpha \right)  \ge  &&- \frac{\alpha}{t} + \frac{t}{n} \left(\left| \nabla f \right|^2 - \partial_t f   \right)^2 \\&&- 2\theta K_1 t \left| \nabla f  \right|^2 - t \theta^2 n^2 \left( K_1 + K_2  \right)^2.
\end{eqnarray*}

And this is the method that we will adopt throughout the paper.




\subsection{Finsler Structures}




\subsubsection{Finsler Metric}
Let $M$ be a $C^\infty$ connected manifold. A Finsler structure on $M$ consists of a $C^\infty$ Finsler norm $F: TM \rightarrow \R$ that satisfies the following conditions\\
\textbf{(F1)} $F$ is $C^\infty$ on $TM \setminus 0$,\\
\textbf{(F2)} $F$ restricted to the fibres is positively $1-$ homogeneous,\\
\textbf{(F3)} For any nonzero tangent vector $\mathbf{y} \in TM$, the approximated symmetric metric tensor, $g_{\mathbf{y}}$, defined by
	\be
		g_{\mathbf{y}} \left( \mathbf{u} , \mathbf{v}  \right) := \frac{1}{2} \frac{\partial^2}{\partial s \partial t} F^2\left( \mathbf{y} + s\mathbf{u} + t\mathbf{v}  \right)|_{s=t=0},
	\ee
is positive definite.




\subsubsection{Cartan Tensor}
One way to measure the non-linearity of a Finsler structure is to introduce the so called, \textbf{Cartan tensor},$C_{\mathbf{y}} : \otimes^3 TM \rightarrow \R$ defined by
	\be
			C_{\mathbf{y}} \left( \mathbf{u} , \mathbf{v} , \mathbf{w}  \right)	:= \frac{1}{2} \frac{d}{dt} \left[ g_{\mathbf{y} + t\mathbf{w}} \left( \mathbf{u} , \mathbf{v} \right) \right].
	\ee




\subsubsection{Legendre Transform}
In order to define the \textbf{gradient} of a function, we need the Legendre transform, $\mathcal{L}^* : T^*M \rightarrow TM$; For $\omega \in T^*M$, $\mathcal{L}^*(\omega)$ is the unique vector $\mathbf{y} \in TM$ such that, 
	\be
			\omega(\mathbf{y}) = F^*(\omega)^2 \;\;and\;\; F(\mathbf{y}) = F^*(\omega),
	\ee
in which, $F^*$ is the dual norm to $F$.

For a smooth function $u: M \rightarrow \R$, the gradient of $u$, $\nabla u(x)$ is defined to be $\mathcal{L}^* \left( Du(x) \right)$.




\subsubsection{Geodesic Spray and Chern Connection and Curvature Tensor}
It is straightforward to observe that the geodesic spray in the Finsler setting is of the form, $\mathbf{G} = y^i \frac{\partial}{\partial x_i} - 2G^i(x,\mathbf{y})\frac{\partial}{\partial y^i}$  where,
	\be
		G^i(x,\mathbf{y}) = \frac{1}{4}g_{\mathbf{y}}^{ik} \left\{ 2\frac{\partial (g_{\mathbf{y}})_{jk}}{\partial x_l} - \frac{\partial (g_{\mathbf{y}})_{jl}}{\partial x_k}   \right\}y^j y^l.
	\ee

The non-linear connection that we will be using in this work is the Chern connection, the connection coefficients of which is given by ($g$ in below is in fact, $g_{\mathbf{y}}$):
	\be	
		\Gamma^i_{jk} = \Gamma^i_{kj} := \frac{1}{2}g^{il} \left\{  \frac{\partial g_{lj} }{ \partial x_k } - \frac{\partial g_{jk}}{ \partial x_l} + \frac{\partial g_{kl} }{ \partial x_j } - \frac{\partial g_{lj} }{ \partial y^r } G^r_k + \frac{\partial g_{jk} }{ \partial y^r} G^r_l - \frac{\partial g_{kl} }{ \partial y^r } G^r_j\right\},
	\ee
where, $G^i_j := \frac{\partial G^i}{\partial y^j}$.

For Berwald metrics, the geodesic coefficients, $G^i$ are quadratic in terms of $y$ (by definition) which immensely simplifies the formula for connection coefficients. In fact for Berwald metrics we have $\Gamma^i_{jk} = \frac{\partial^2 G^i}{\partial y^j \partial y^k}$.

Similar to the Riemannian setting, one uses the Chern connection (and the associated covariant differentiation) to define the curvature tensor (of course depending on a nonzero vector field $V$)
	\be 
		R^V(X,Y)Z := [\nabla^V_X , \nabla^V_Y]Z - \nabla^V_{[X,Y]}Z.
	\ee




\subsubsection{Flag and Ricci Curvatures}
\textbf{Flag curvature} is defined similar to the sectional curvature in the Riemannian setting. For a fixed flag pole, $\mathbf{v} \in T_xM$	and for $\mathbf{w} \in T_xM$, the flag curvature is defined as
	\be 
	\mathcal{K^\mathbf{v}\left( \mathbf{v} , \mathbf{w} \right)} := \frac{g_{\mathbf{v}}\left( R^\mathbf{v}(\mathbf{v},\mathbf{w})\mathbf{w} , \mathbf{v} \right)}{g_{\mathbf{v}}\left(\mathbf{v}  , \mathbf{v} \right)g_{\mathbf{v}}\left(\mathbf{w}  , \mathbf{w} \right) - g_{\mathbf{v}}\left(\mathbf{v}  , \mathbf{w} \right)^2}.
	\ee	
	
The Ricci curvature is then the trace of the Flag curvature i.e.
\be
			Ric(\mathbf{v}) := F^2(\mathbf{v}) \sum_{i=1}^{n-1} \mathcal{K^\mathbf{v}\left( \mathbf{v} , \mathbf{e_i} \right)},
	\ee	
where, $\{\mathbf{e_1} , \dots , \mathbf{e_{n-1} , \frac{\mathbf{v}}{F(\mathbf{v})}} \}$ constitutes a $g_{\mathbf{v}}-$orthonormal basis of $T_xM$.




\subsubsection{Akbarzadeh's Ricci Tensor, $Ric_{ij}$}

Akbarzadeh's Ricci tensor is defined as follows
\be
	Ric_{ij} := \frac{\partial^2}{\partial y^i  \partial y^j}\left( \frac{Ric}{2}   \right).
\ee

It can be shown that the scalar Ricci curvature, $Ric$ and Akbarzadeh's Ricci tensor. $Ric_{ij}$ have the same geometrical implications. For further details regarding this tensor, see~ \cite{Bao-Robles}.




\subsubsection{$\mathbf{S}-$Curvature}\label{sec:S-curv}

Associated to any Finsler structure, there is one canonical measure, called the Busemann-Hausdorff measure, given by
\be
	dV_F := \sigma_F(x)dx_1 \wedge \dots \wedge dx_n,
\ee
where $\sigma_F(x)$ is the volume ratio
\be\label{eq-sigmaF-defn}
	\sigma_F(x) := \frac{vol \left(B_{\R^n}(1) \right)}{vol \left( \mathbf{y} \in T_xM \;|\; F(\mathbf{y}) < 1 \right)};
\ee
The set whose volume comes in the denominator of (\ref{eq-sigmaF-defn}) is called the indicatrix and there is often no known way to express its volume in terms of the equation of $F$.

The $S-$curvature, which is another measure of non-linearity, is then defined as
	\be
		\mathbf{S}(\mathbf{y}) := \frac{\partial G^i}{\partial y^i}(x , \mathbf{y}) - y^i \frac{\partial}{\partial x_i}\left( \ln \sigma_F(x) \right).
	\ee
For more details please see~\cite{Shen-S-curv} for example.




\subsubsection{Hessian, Divergence and Laplacian}
The Hessian in a Finsler structure is defined as
	\be
		Hess(u) \left(X ,Y \right) := XY(u) - \nabla_X^{\nabla u}Y (u) = g_{\nabla u} \left( \nabla_X^{\nabla u} \nabla u , Y \right).
	\ee

As usual, for a twice differentiable function $u$,
\begin{eqnarray}
	Hess(u)\left( \frac{\partial}{\partial x_i} , \frac{\partial}{\partial x_j} \right) = \frac{\partial^2 u}{\partial x_i \partial x_j}  - \Gamma_{ij}^k \frac{\partial u}{\partial x_k}
\end{eqnarray}

For a smooth measure $\mu = e^{-\Psi}dx_1 \wedge \dots \wedge dx_n$ and a vector field $V$, the divergence is defined as follows
	\be
		\operatorname{div}_\mu V:= \sum_{i=1}^{n}	\left(\frac{\partial V_i}{\partial x_i} - V_i \frac{\partial \Psi}{\partial x_i}\right).
	\ee

Now using this divergence, one can define the distributional Laplacian of a function $u \in H^1(M)$ by $\Delta u := \operatorname{div}_\mu \left( \nabla u \right)$, ie.
	\be
		\int_M \phi\Delta u d\mu := - \int_M D\phi \left(\nabla u  \right)	d\mu.	
	\ee
for $\phi \in C^\infty(M)$.	

The Finsler distributional Laplacian is non-linear but fortunately there is a way to relate it to the trace of the Hessian by adding a $\mathbf{S}-$curvature term. Indeed, one has
 	\be\label{eq-laplace-trace-Hessian}
 		\Delta u = \tr_{\nabla u} Hess(u) - \mathbf{S}\left( \nabla u \right).
	\ee
For a proof of (~\ref{eq-laplace-trace-Hessian}), see for instance~\cite{ComparisonFinsler}.







\subsection{ Weighted Ricci Curvature and Bochner-Weitzenb\"{o}ck Formula}\label{subsection-OS-BWformula}

The notion of the \textbf{weighted Ricci curvature}, $Ric_N$, of a Finsler structure equipped with a measure, $\mu$, was introduced by Ohta~\cite{Ohta-interpolation}. Take a unit vector $\mathbf{v} \in T_xM$ and let $\gamma : [-\epsilon , + \epsilon] \rightarrow M$ be a short geodesic whose velocity at time $t=0$, is $\dot{\gamma}(0) = \mathbf{v}$. Decompose the measure, $\mu$ along $\gamma$ with respect to the Riemannian volume form i.e. let $\mu = e^{-\Psi}dvol_{\dot{\gamma}}$; Then,
\begin{align}
& Ric_n(\mathbf{v}) := \begin{cases} Ric\left( \mathbf{v} \right) + \left( \Psi \circ \gamma \right)''(0) \;\; if \;\; \left( \Psi \circ \gamma \right)'(0) = 0, \\  - \infty \;\; otherwise, \end{cases} \\  & Ric_N(\mathbf{v}) := Ric\left( \mathbf{v} \right) + \left( \Psi \circ \gamma \right)''(0) - \frac{\left( \Psi \circ \gamma \right)'(0)^2}{N - n } \;\;when
;\; n < N < \infty, \\ & Ric(\mathbf{v}) := Ric\left( \mathbf{v} \right) + \left( \Psi \circ \gamma \right)''(0).
\end{align}
Also $Ric_N(\lambda \mathbf{v} ) := \lambda^2 Ric_N(\mathbf{v})$ for $\lambda \ge 0$.

It is proven in Ohta~\cite{Ohta-interpolation} that the curvature bound $Ric_N \ge K F^2$ is equivallent to Lott-Villani-Sturm's $CD(K,N)$ condition.

Using the weighted Ricci curvature bounds,  Ohta-Sturm ~\cite{Ohta-Sturm-Bochner} proved the Bochner-Weitzenb\"{o}ck formulae (both point-wise and integrated versions) for Finsler structures. For $u \in C^\infty(M)$, the point-wise version of the identity and inequality go as
\begin{align}
		& \Delta^{\nabla u} \left( \frac{F^2 \left( \nabla u \right)}{2}  \right) - D \left( \Delta u\right)\left( \nabla u  \right) = Ric_\infty\left(  \nabla u \right) +  \|  \nabla^2 u\|_{HS(\nabla u)}^2 \;\; (identity),\\
		& 		\Delta^{\nabla u} \left( \frac{F^2 \left( \nabla u \right)}{2}  \right) - D \left( \Delta u\right)\left( \nabla u  \right)  \ge  Ric_N \left(  \nabla u \right) + \frac{\left( \Delta u \right)^2}{N} \;\; (inequality).
\end{align}






\section{Estimates}\label{sec-estimates}

In this section we will gather all the required lemmas and estimates that we will need before being able to apply the maximum principle.




\subsection{Evolution of the Legendre Transform}
Since in Finsler setting the gradient is non-linear and depends on the Legendre transform we will need to know the evolution of the Legendre transform under FRF.

Let $\left(M,F \right)$ be a finsler structure evolving under Finsler-Ricci flow. Then the inverse of the Legendre transform , $(\mathcal{L}^*)^{-1}: TM  \to T^*M $ is defined as
\be
	(\mathcal{L}^*)^{-1}(x,y) = (x,p)\;\;where,\;\;p_i = g_{ij}(x,y)y^j.
\ee

To explicitly formulate the Legendre transform, for any given $\omega \in T^*_xM$, we have $\mathcal{L}^*(\omega) = y \in T_xM$, where $y$ is the unique solution to the following \textit{non-linear} system
\be
\begin{cases} g(x,y)_{11} . y^1  + \dots + g(x,y)_{1n}.y^n = \omega_1 \\ \dots \\ g(x,y)_{n1} . y^1  + \dots + g(x,y)_{nn}.y^n = \omega_n, \end{cases} 
\ee
or in the matrix form we have
\be\label{eq-legendre-defn}
	g(\mathbf{y})\mathbf{y} = \omega.
\ee

\begin{lem}
	Let $\left(M,F(t) \right)$ be a Finsler structure evolving by FRF; Then, the Legendre transform $\mathcal{L}^*: T^*M  \rightarrow TM$ satisfies
	\be
		\partial_t \mathcal{L}^* = 2Ric^i_j \mathcal{L}^*,
	\ee
i.e. for any fixed 1-form $\omega$ with $\mathcal{L}^*(\omega) = \mathbf{y} = y^i \frac{\partial}{\partial x_i} \in TM$, we have 
\be
	\partial_t y^i = 2Ric^i_ry^r,
\ee
where, $Ric^i_r := g^{ij}Ric_{jr}$.
\end{lem}

\begin{proof}
	Fix $\omega$ and differentiate both sides of (\ref{eq-legendre-defn}) w.r.t. $t$ to get
	\be
	[\partial_t g(\mathbf{y})]\mathbf{y} + g(\mathbf{y}) \partial_t \mathbf{y} = 0.
\ee
Therefore, 
\be\label{eq:leg-evol1}
	\partial_t \mathbf{y} = - g(\mathbf{y})^{-1}\partial_t g(\mathbf{y})\mathbf{y},
\ee
 Expanding the RHS of (\ref{eq:leg-evol1}), for every $i$, we have
\begin{eqnarray}\label{eq:leg-evol2}
	\partial_t y^i &=& - g(\mathbf{y})^{ij}\left(\partial_t g(\mathbf{y})\right)_{jr}y^r \notag \\ &=& 2  g(\mathbf{y})^{ij}Ric_{jr}(\mathbf{y})y^r - g(\mathbf{y})^{ij}\left(\frac{\partial g_{jr}}{\partial y^k}\partial_t y^k \right)y^r \\  &=& 2Ric^i_r(\mathbf{y})y^r. \notag
\end{eqnarray}

Notice that, the second term in the second line of (\ref{eq:leg-evol2}) vanishes by Euler's theorem. 
\end{proof}
 



\subsection{Evolution of $F^2(\nabla f)$}

One crucial step in the proof of gradient estimates is to be able to estimate the evolution of the term $F^2(\nabla f)$.
\begin{lem}\label{lem:evol-norm-sq1}
Let $\left(M,F(t) \right)$ be a time dependent Finsler structure, then
\be
	\partial_t \left[ F^2 \left( \nabla f \right) \right] = 2g^{ij}\left( Df \right)[ \partial_t f]_if_j + \left[ \partial_t g^{ij}\right](Df)f_if_j.
\ee
\end{lem}
\begin{proof}
Simple differentiation gives
\begin{eqnarray}\label{eq:evol-norm-squared1}
		\partial_t \left[ F^2 \left( \nabla f \right) \right] &=&  \partial_t \left[ F^* \left( D f \right)^2 \right] \notag \\  &=& \partial_t \left[ g^{ij}(Df)f_if_j \right] \\ &=& 2g^{ij}\left( Df \right)[ \partial_t f]_if_j + \partial_t \left[ g^{ij}(Df) \right ]f_if_j. \notag
\end{eqnarray}
Expanding the second term of the last line in (\ref{eq:evol-norm-squared1}) , we have
\be\label{eq:evol-norm-squared2}
	\partial_t \left[ g^{ij}(Df) \right ]f_if_j =  \left[ \partial_t g^{ij} \right ] (Df)f_if_j + \frac{\partial g^{ij}}{\partial y^k} \partial_t y^k (Df)f_if_j.
\ee
Using Euler's theorem, the second term of the RHS of (\ref{eq:evol-norm-squared2}) vanishes.
\end{proof}

\begin{lem}\label{lem:evol-norm-sq2}
Suppose $F$ is evolving by Finsler-Ricci flow equation, then
\be\label{eq:evol-norm-squared3}
	\partial_t \left[ F^2 \left( \nabla f \right) \right] = 2D(\partial_t f)(\nabla f) + 2 Ric^{ij}(Df)f_if_j.
\ee
\end{lem}
\begin{proof}
It is standard to see that under Finsler-Ricci flow, 
\be
	\partial_t g^{ij} = 2 Ric^{ij},
\ee
where, as before, $Ric^{ij} := g^{ir}g^{js}Ric_{rs}$.
\end{proof}






\section{Proof of Main Theorem}

In this section we will complete the proof of our main theorem. Throughout the rest of these notes, we have a solution, $u$, of the heat equation. The Laplacian, gradient and the Legendre transform are all with respect to $V := \nabla u$ and are valid on $M_u := \{x \in M \; | \; \nabla u(x) \neq 0\}$.

Let $\sigma(t,x) = t \partial_t f(t,x)$ where, $f = \log u$. Then we have $g_{\nabla f} = g_V$. Let
\be
	\alpha(t,x):= t \left\{ F^2 \left( \nabla f(t,x) \right) - \theta \partial_t f(t,x)   \right\} = t  F^2 \left( \nabla f(t,x) \right)  - \theta \sigma.
\ee

\begin{lem}
In the sense of distributions, $\sigma(t,x)$ satisfies the parabolic differential equality,
\be\label{eq-sigma}
	\Delta \sigma - \partial_t \sigma + \frac{\sigma}{t} + 2D\sigma(\nabla f) = t \left\{ - 2 Ric^{ij}(\nabla f)  f_if_j - 2\left( Ric \right)^{kl}\left(\nabla f \right)f_{kl} \right\}.
\ee
\end{lem}

\begin{proof}
We first note that for any non-negative test function $\phi \in H^1([0,T] \times M)$ whose support is included in the domain of the local coordinate, 
\be\label{eq-sigma-1}
\partial_t \left( D(t\phi)(\nabla f) \right) = D\left( \partial_t(t \phi)\right)(\nabla f) + D(t \phi)\left(\nabla (\partial_t f) \right) + 2\left( Ric \right)^{ij}\left( \nabla f \right)\frac{\partial (t \phi)}{\partial x_i}\frac{\partial f}{\partial x_j}.
\ee
Indeed,
\begin{eqnarray}
\partial_t \left( D(t\phi)(\nabla f) \right) &=& D\left( \partial_t(t \phi)\right)(\nabla f) + D(t \phi)\left(\partial_t (\mathcal{L}^* (Df) \right) \notag \\ &=& D\left( \partial_t(t \phi)\right)(\nabla f) + D(t \phi)\Big(\partial_t (\mathcal{L}^*)(Df) + \mathcal{L}^* (D \partial_t f) \Big) \notag \\&=& D\left( \partial_t(t \phi)\right)(\nabla f) + D(t \phi)\Big(\partial_t (\mathcal{L}^*)(Df)\Big) + D(t \phi)\Big( \mathcal{L}^* (D \partial_t f) \Big) \\&=& D\left( \partial_t(t \phi)\right)(\nabla f) + D(t \phi)\left(\nabla (\partial_t f) \right) + 2g^{sj}\left( Ric \right)^i_s\left( \nabla f \right)\frac{\partial (t \phi)}{\partial x_i}\frac{\partial f}{\partial x_j}. \notag  \\&=& D\left( \partial_t(t \phi)\right)(\nabla f) + D(t \phi)\left(\nabla (\partial_t f) \right) + 2\left( Ric \right)^{ij}\left( \nabla f \right)\frac{\partial (t \phi)}{\partial x_i}\frac{\partial f}{\partial x_j}. \notag
\end{eqnarray}
That is
\be\label{eq:eq-sigma-proof-0}
	- D(t \phi)\left(\nabla (\partial_t f) \right) = - \partial_t \left( D(t\phi)(\nabla f) \right) + D\left( \partial_t(t \phi)\right)(\nabla f) + 2\left( Ric \right)^{ij}\left( \nabla f \right)\frac{\partial (t \phi)}{\partial x_i}\frac{\partial f}{\partial x_j}.
\ee

Multiplying the LHS of (\ref{eq-sigma}) by $\phi$,  integrating and then substituting (\ref{eq:eq-sigma-proof-0}), we get
\begin{eqnarray}\label{eq-sigma-proof-1}
\mathbf{A} &=& \int_0^T \int_M \left\{ - D\phi (\nabla \sigma) + \partial_t\phi . \sigma + \frac{\phi \sigma}{t} + 2\phi D\sigma(\nabla f)\right\} dm dt \notag \\ &=& \int_0^T \int_M \left\{ - D(t\phi) \left( \nabla (\partial_t f) \right) + \partial_t (t \phi) \partial_t f + 2 t \phi D (\partial_t f )(\nabla f)\right\} dm dt \notag \\ &=& \int_0^T \int_M \left\{  D \left( \partial_t (t\phi) \right) \left( \nabla f \right) + \partial_t (t \phi) \left( \Delta f + F^2(\nabla f)   \right) \right. \\ &&  \phantom{sajjadlakz} \left.  + 2\left( Ric \right)^{ij}\left( \nabla f \right)\frac{\partial (t \phi)}{\partial x_i}\frac{\partial f}{\partial x_j} + 2 t \phi D (\partial_t f )(\nabla f) \right\}  dm dt.\notag
\end{eqnarray}

Using the estimates we have obtained for $\partial_t \left[ F(\nabla f)^2  \right]$ in Lemmas \ref{lem:evol-norm-sq1} and \ref{lem:evol-norm-sq2}, we arrive at
\begin{eqnarray}\label{eq:sigma-proof-2}
\mathbf{A} &=& \int_0^T \int_M \left\{  D \left( \partial_t (t\phi) \right) \left( \nabla f \right) + \partial_t (t \phi) \left( \Delta f\right) + \partial_t (t \phi) \left(F^2(\nabla f)   \right) \right. \notag \\&&   \phantom{sajjadlakz} \left. \qquad +  2\left( Ric \right)^{ij}\left( \nabla f \right)\frac{\partial (t \phi)}{\partial x_i}\frac{\partial f}{\partial x_j} + 2 t \phi D (\partial_t f )(\nabla f) \right\}  dm dt \notag \\ &= &  \int_0^T \int_M  \left\{ \partial_t (t \phi) \left(F^2(\nabla f)  \right)+ 2\left( Ric \right)^{ij}\left( \nabla f \right)\frac{\partial (t \phi)}{\partial x_i}\frac{\partial f}{\partial x_j} + t \phi \partial_t \left[ F(\nabla f)^2 \right] \right.\\&&\phantom{sajjadlakz} \left. - 2t \phi Ric^{ij}(\nabla f)  f_if_j \right\} dmdt  \notag \\&= & \int_0^T \int_M t\phi \left\{ -2 Ric^{ij}(\nabla f)  f_if_j - 2Ric^{ij}(\nabla f)f_{ij} \right\} dmdt. \notag
\end{eqnarray}
Notice that the Euler's theorem has been used in the last line of (\ref{eq:sigma-proof-2}).
\end{proof}
Now we can proceed to compute a parabolic partial differential inequality for $\alpha(t,x)$ with a similar LHS.

\begin{lem}
	In the sense of distributions, $\alpha(t,x)$ satisfies
	\be
	\Delta^V \alpha + 2D\alpha(\nabla f) - \partial_t \alpha  + \frac{\alpha}{t}  = \mathbf{B},
\ee
where,
\begin{eqnarray}
  \mathbf{B} = \theta \left( 2t Ric^{ij}(\nabla f)f_i f_j  +  2t Ric^{kl}(\nabla f)f_{kl} \right) + 2t Ric(\nabla f) +2t \| \nabla^2 f \|_{HS(\nabla f)}^2 - 2t Ric^{ij}(\nabla f)f_i f_j .\notag
\end{eqnarray}
\end{lem}

\begin{proof}
For a non-negative test function, $\phi$, one computes
\begin{eqnarray}
	&&\int_0^T \int_M \left\{ - D\phi (\nabla \alpha) + \partial_t\phi . \alpha + \frac{\phi \alpha}{t} + 2\phi D\alpha(\nabla f)\right\} dm dt \notag \\ && \qquad = -\theta \mathbf{A} + \int_0^T \int_M \left\{ - tD\phi \left( \nabla \left( F^2(\nabla f) \right) \right) + \partial_t\phi \left(t F^2(\nabla f) \right)  \right. \notag \\ && \phantom{sajjadlakz} \left. \qquad \qquad \qquad + \phi \left( F^2(\nabla f) \right) + 2t \phi D\left( F^2(\nabla f) \right)(\nabla f)\right\} dm dt \notag \\ && \qquad = -\theta \mathbf{A} + \int_0^T \int_M \left\{ - tD\phi \left( \nabla \left( F^2(\nabla f) \right) \right) - \phi . \partial_t \left( t\left( F^2(\nabla f) \right) \right)    \right. \\ &&\phantom{sajjadlakz} \left. \qquad \qquad \qquad + \phi \left( F^2(\nabla f) \right) + 2t \phi D\left( F^2(\nabla f) \right)(\nabla f)\right\} dm dt \notag \\ && \qquad = -\theta \mathbf{A} + \int_0^T \int_M \left\{ - tD\phi \left( \nabla \left( F^2(\nabla f) \right) \right) - \phi .  \left( F^2(\nabla f) + t\partial_t \left( F^2(\nabla f) \right) \right)  \right. \notag \\ && \phantom{sajjadlakz} \left. \qquad \qquad \qquad + \phi \left( F^2(\nabla f) \right) + 2t \phi D\left( F^2(\nabla f) \right)(\nabla f)\right\} dm dt  \notag \\&& \qquad = -\theta \mathbf{A} + \int_0^T \int_M \left\{ - tD\phi \left( \nabla \left( F^2(\nabla f) \right) \right) - \phi .   t\partial_t \left( F^2(\nabla f)  \right) \right. \notag \\ && \phantom{sajjadlakz} \left. \qquad \qquad \qquad + 2t \phi D\left( F^2(\nabla f) \right)(\nabla f)\right\} dm dt, \notag
\end{eqnarray}
where, $\mathbf{A}$ is as in (\ref{eq:sigma-proof-2}).

Again, using the estimates for $\partial_t \left[ F(\nabla f)^2 \right]$ (as in Lemmas \ref{lem:evol-norm-sq1} and \ref{lem:evol-norm-sq2}), we arrive at
\begin{eqnarray}
	 &&\int_0^T \int_M \left\{ - D\phi (\nabla \alpha) + \partial_t\phi . \alpha + \frac{\phi \alpha}{t} + 2\phi D\alpha(\nabla f)\right\} dm dt \notag \\ &&\qquad = -\theta \mathbf{A} + \int_0^T \int_M \left\{ - tD\phi \left( \nabla \left( F^2(\nabla f) \right) \right) - \phi .   t\partial_t \left( F^2(\nabla f)  \right) \right. \notag  \\ && \phantom{sajjadlakz} \left. \qquad \qquad \qquad + 2t \phi D\left( F^2(\nabla f) \right)(\nabla f)\right\} dm dt \notag \\&& \qquad = -\theta \mathbf{A} + \int_0^T \int_M \left\{ - tD\phi \left( \nabla \left( F^2(\nabla f) \right) \right) - 2t\phi D(\partial_t f)(\nabla f) \right. \notag \\ && \phantom{sajjadlakz} \left. \qquad \qquad \qquad -  2t\phi Ric^{ij}(\nabla f)f_if_j + 2t \phi D\left( F^2(\nabla f) \right)(\nabla f)\right\} dm dt \\&& \qquad = -\theta \mathbf{A} + \int_0^T \int_M \left\{ - tD\phi \left( \nabla \left( F^2(\nabla f) \right) \right) - 2t\phi D(\Delta f)(\nabla f) -2t\phi  D\left(F^2(\nabla f) \right) (\nabla f) \right. \notag \\ && \phantom{sajjadlakz} \left. \qquad \qquad \qquad -   2t\phi Ric^{ij}f_if_j+ 2t \phi D\left( F^2(\nabla f) \right)(\nabla f)\right\} dm dt \notag \\&& \qquad = -\theta \mathbf{A} + \int_0^T \int_M \left\{ - tD\phi \left( \nabla \left( F^2(\nabla f) \right) \right) - 2t\phi D(\Delta f)(\nabla f)  - 2 t\phi Ric^{ij}(\nabla f)f_if_j\right\} dm dt. \notag
\end{eqnarray}

By applying the Bochner-Weitzenb\"{o}ck formula (proven in Ohta-Sturm~\cite{Ohta-Sturm-Bochner}, see also section~\ref{subsection-OS-BWformula} here) and noticing that $\mathbf{S}=0$ implies $Ric_{\infty}(\mathbf{v}) = Ric(\mathbf{v})$; we can continue as follows.
\begin{eqnarray}
 -\theta \mathbf{A} &+& \int_0^T \int_M \left\{ - tD\phi \left( \nabla \left( F^2(\nabla f) \right) \right) - 2t\phi D(\Delta f)(\nabla f)  - 2t\phi Ric^{ij}f_i f_j\right\} dm dt \notag \\& = & -\theta \mathbf{A} +  \int_0^T \int_M \phi \left\{2t Ric(\nabla f) +2t \| \nabla^2 f \|_{HS(\nabla f)}^2 - 2t Ric^{ij}(\nabla f)f_i f_j \right\}  dm dt \notag. 
\end{eqnarray}

Now, substituting $\mathbf{A}$ from (\ref{eq-sigma-proof-1}), we have
\begin{eqnarray}
  \mathbf{B} = \theta \left( 2t Ric^{ij}(\nabla f)f_i f_j  +  2t Ric^{kl}(\nabla f)f_{kl} \right) + 2t Ric(\nabla f) +2t \| \nabla^2 f \|_{HS(\nabla f)}^2 - 2t Ric^{ij}(\nabla f)f_i f_j ,\notag
\end{eqnarray}
\end{proof}

Now we have all the ingredients that we need to complete the proof of Theorem~\ref{thm-main}.

\subsection*{Proof of Theorem~\ref{thm-main}}

Assume the curvature bounds as in the statement of Theorem~\ref{thm-main}, and that the $\mathbf{S}-$curvature vanishes. The constants obtained below all depend on our curvature bounds and  ellipticity of the flow.

Lets start with $B(t,x)$,
\begin{eqnarray}
	B(t,x)  = \theta \left( 2t Ric^{ij}(\nabla f)f_i f_j  +  2t Ric^{kl}(\nabla f)f_{kl} \right) + 2t Ric(\nabla f) +2t \| \nabla^2 f \|_{HS(\nabla f)}^2 - 2t Ric^{ij}(\nabla f)f_i f_j. \notag
\end{eqnarray}

Young's inequality tells us that:
\begin{eqnarray}
	| Ric^{kl}f_{kl} | \le \frac{\theta}{2}(Ric^{kl})^2 + \frac{1}{2\theta} f_{kl}^2,
\end{eqnarray}
therefore,
\begin{eqnarray}
	2\theta t  | Ric^{kl}f_{kl} | \le  t \theta^2 (Ric^{kl})^2 + t f_{kl}^2.
\end{eqnarray}

Pick a normal coordinate system with respect to $g_{\nabla f}$ with $\nabla f(x) = \frac{\partial}{\partial x_1}$ as well as $\Gamma^1_{ij}\left( \nabla f(x) \right) = 0$ for all $i,j$. Then,
\be
	Ric^{ij}(\nabla f) = Ric_{ij}(\nabla f) \;\; and\;\; \| \nabla^2 f \|_{HS(\nabla f)}^2 = \sum f_{ij}^2 \;\; and \;\; \sum_{i=1}^n f_{ii} = \Delta f(x)
\ee
and consequently,
\begin{eqnarray}
	B(t,x) &\ge & 2t\theta Ric_{ij}(\nabla f)f_i f_j  - t \sum \theta^2 (Ric_{kl})^2 - t \sum f_{kl}^2 \notag \\ &&+ 2t Ric(\nabla f) +2t \| \nabla^2 f \|_{HS(\nabla f)}^2 - 2t Ric_{ij}(\nabla f)f_i f_j  \\ & \ge & - 2t\theta K_1 F^2(\nabla f) - 2tK_1 F^2(\nabla f) + t \sum f_{ij}^2 -t \theta^2 n^2 C_2 + 2tK_1F^2(\nabla f) . \notag
\end{eqnarray}

On the other hand, one computes
\be
	\sum f_{ij}^2 \ge \sum f_{ii}^2 \ge 1/n \left(\sum f_{ii} \right)^2 = \frac{1}{n} \left( \Delta f \right)^2.
\ee
hence,
\begin{eqnarray}
		 t \sum f_{ij}^2 \ge  \frac{t}{n} \left( \Delta f \right)^2.
\end{eqnarray}

Putting all the above estimates together, and noting that $\theta >1$, we get

\begin{eqnarray}
	B(t,x) & \ge & \frac{t}{n} \left( \Delta f \right)^2 - 2t\theta K_1 F^2(\nabla f) - 2tK_1F^2(\nabla f)  -t \theta^2 n^2 C_2 + 2tK_1F^2(\nabla f)  \notag \\ & \ge & \frac{t}{n}  \left( \Delta f \right)^2 - 2t\theta K_1F^2(\nabla f) -t \theta^2 n^2 C_2 .
\end{eqnarray}

Replacing the term, $\Delta f$ with $\left( F(\nabla f)^2 - \partial_t f \right)$, we get the inequality
\begin{eqnarray}
	B(t,x) \ge  \frac{t}{n}\left( F(\nabla f)^2 - \partial_t f \right)^2 - 2t\theta C_1F^2(\nabla f) - t\theta^2 n^2 C_2,
\end{eqnarray}
where,
\be\label{eq:constants}
	C_1 = K_1 \;\;and \;\; C_2 = \max\{K_1^2 , K_2^2  \}.
\ee
This means that,
\be\label{eq:main}
\Delta^V \alpha + 2D\alpha(\nabla f) - \partial_t \alpha \ge - \frac{\alpha}{t} + \frac{t}{n}\left( F(\nabla f)^2 - \partial_t f \right)^2 - 2t\theta C_1F^2(\nabla f) - t\theta^2 n^2 C_2. 
\ee

This inequality is exactly of the form that appears in~\cite{LIU-Grad-Est} and a computation similar to the one at the end of the proof of in~\cite[Theorem 2]{LIU-Grad-Est} (using quadratic formula and maximum principle), gives the desired result. For the sake of clarity, we will repeat the computation here.

Let
\be
	\bar{\alpha} := \alpha -  t\frac{n\theta^3 C_1}{(\theta -1)} - tn^{3/2}\theta^2 \sqrt{C_2}.
\ee
Suppose the maximum of $\bar{\alpha}$ is attained at $(x_0,t_0)$ and $\bar{\alpha} (x_0,t_0) > n\theta^2 $ (implicitly implies $t_0>0$). Therefore, at $(x_0,t_0)$, we have
\be
	0 \ge (\Delta - \partial_t) \bar{\alpha} \ge (\Delta - \partial_t) \alpha .
\ee

Let $w := F^2(\nabla f)$ and $z:= \partial_t f $, then in terms of $w$ and $z$, we have
\be
	0 \ge -\frac{\alpha}{t_0} + \frac{t_0}{n}(w -z)^2 - 2t_0 \theta C_1 w - t_0 \theta^2 n^2 C_2.
\ee
By the quadratic formula, we get
\begin{eqnarray}
	&&  \frac{t_0}{n}(w -z)^2  - 2t_0\theta C_1 w = \phantom{sajjad} \notag \\ && \phantom{sajjad}\frac{t_0}{n}\left( \frac{1}{\theta^2}(w - \theta z)^2 + \left( \frac{\theta - 1}{\theta} \right)^2 w^2 - 2\theta n C_1 w + 2 \left( \frac{\theta - 1}{\theta^2}w \right) (w - \theta z) \right)  \\ && \phantom{sajjadsajjadsajjadsajjadsaj} \ge  \frac{t_0}{n}\left( \frac{1}{\theta^2}(w - \theta z)^2 -  \frac{\theta^4 n^2 C_1^2}{(\theta -1)^2} + 2 \left( \frac{\theta - 1}{\theta^2}w \right) (w - \theta z)  \right). \notag
\end{eqnarray}
Therefore,
\begin{eqnarray}\label{eq:quadratic1}
	0 &\ge & \frac{t_0}{n \theta^2}\left( \frac{\alpha}{t_0} \right)^2 - \frac{\alpha}{t_0} - \frac{n \theta^4  C_1^2}{(\theta -1)^2}t_0  - t_0\theta^2 n^2 C_2 + \frac{2t_0}{n}\frac{\theta - 1}{\theta^2}F^2(\nabla f)\left( \frac{\alpha}{t_0} \right) \notag \\ &\ge & \frac{t_0}{n \theta^2}\left( \frac{\alpha}{t_0} \right)^2 - \frac{\alpha}{t_0} - \frac{n \theta^4  C_1^2}{(\theta -1)^2}t_0  - t_0\theta^2 n^2 C_2. 
\end{eqnarray}

Using the quadratic formula one more time, (\ref{eq:quadratic1}) implies that
\be
	\frac{\alpha}{t_0} \le \frac{n \theta^2}{t_0} + \frac{n \theta^3 C_1}{\theta - 1} + n^{\frac{3}{2}}\theta^2 \sqrt{C_2};
\ee
which in turn, implies
\be
	\bar{\alpha}(x_0 , t_0) \le n \theta^2 ,
\ee
and this is a contradiction. Therefore, 
\be
	F^2\left( \nabla \left( \log u \right) (t,x)\right) - \theta \partial_t \left( \log u \right)(t,x) \le  \frac{n \theta^2}{t} + \frac{n\theta^3 C_1}{(\theta -1)} + n^{3/2}\theta^2 \sqrt{C_2},
\ee
with $C_1$ and $C_2$ as in (\ref{eq:constants}).\\
QED.

\subsection*{Proof of Corollary~\ref{cor-global}}
From Theorem \ref{thm-main}, we know
\be
	F^2\left( \nabla \left( \log u \right) (t,x)\right) - \theta \partial_t \left( \log u \right)(t,x) \le \frac{ n \theta^2}{t} + C(n, \theta )\left(C_1 + \sqrt{C_2}  \right).
\ee

Let $l(s) := \ln u \left(\gamma(s) , \tau(s) \right) =  f\left(\gamma(s) , \tau(s) \right) $. Then,
	\begin{eqnarray}
		\frac{\partial l(s)}{\partial s} &&= \left( t_2 - t_1 \right) \left( \frac{Df \left( \dot{\gamma}(s) \right)}{t_2 - t_1} - \partial_t f \right) \notag \\ &&\le \left( t_2 - t_1 \right) \left( \frac{F\left(\nabla f \right)F\left(\dot{\gamma}\right)}{t_2 - t_1} - \partial_t f \right) \notag \\ &&\le \left( t_2 - t_1 \right)\left( \frac{\epsilon F^2\left(\dot{\gamma}\right)|_\tau}{2 \left(t_2 - t_1  \right)^2} + \frac{1}{2\epsilon}F^2\left(\nabla f \right) - \partial_t f \right) \\&&\le  \frac{\epsilon F^2\left(\dot{\gamma}\right)|_\tau}{2 \left(t_2 - t_1  \right)} + \left( t_2 - t_1 \right)\left( \frac{ 2n\epsilon}{\tau} + C(n , \epsilon) \left( C_1 + \sqrt{C_2}  \right) \right). \notag
	\end{eqnarray}
Integrating this inequality gives
\begin{eqnarray}
		\ln \frac{u\left( x,t_1 \right)}{u\left( y,t_2 \right)}&=& \int_0^1 \; \frac{\partial l(s)}{\partial s} ds \notag \\ &\le &  \int_0^1 \;\frac{\epsilon F^2\left(\dot{\gamma}\right)|_\tau }{2 \left(t_2 - t_1  \right)} ds + C(n , \epsilon)(t_2 - t_1)\left(C_1 + \sqrt{C_2} \right) + 2\epsilon n \ln \frac{t_2}{t_1}.
\end{eqnarray}

QED.

\bibliographystyle{plain}
\bibliography{reference2015}

\end{document}